\documentclass[reqno]{amsart}

\theoremstyle{plain}
\newtheorem{theorem}{Theorem} [section]
\newtheorem{corollary}[theorem]{Corollary}
\newtheorem{lemma}[theorem]{Lemma}
\newtheorem{proposition}[theorem]{Proposition}

\theoremstyle{definition}
\newtheorem{definition}[theorem]{Definition}

\newtheorem{remark}[theorem]{Remark}

\newtheorem{example}[theorem]{Example}

\def\C{{\mathbb{C}}}
\def\N{{\mathbb{N}}}
\def\R{{\mathbb{R}}}
\def\Z{{\mathbb{Z}}}

\def\H{{\mathcal{H}}}
\def\F{{\mathcal{F}}}

\newcommand{\w}{\omega}

\newcommand{\f}{\alpha}
\newcommand{\ti}{\theta}
\newcommand{\g}{\gamma}
\newcommand{\sn}{\overline{\text{span}}}

\newcommand{\A}{\mathcal{A}}
\newcommand{\B}{\mathcal{B}}
\newcommand{\G}{\mathcal{G}}

 \theoremstyle{plain}
 \theoremstyle{plain}
 \theoremstyle{plain}
 \theoremstyle{definition}
 
 \theoremstyle{remark}

\newtheorem{thm*}{Teorema}[section]

\newcommand{\CHI}{\hbox{\raise .4ex \hbox{$\chi$}}}

\def\subset{\subseteq}

\begin{document}

\title{Sampling in a Union of Frame Generated Subspaces}

\author[M. Anastasio and  C. Cabrelli]{Magal\'i Anastasio and  Carlos Cabrelli}

\address{\textrm{(M. Anastasio)}
Departamento de Matem\'atica,
Facultad de Ciencias Exac\-tas y Naturales,
Universidad de Buenos Aires, Ciudad Universitaria, Pabell\'on I,
1428 Buenos Aires, Argentina and
CONICET, Consejo Nacional de Investigaciones
Cient\'ificas y T\'ecnicas, Argentina}
\email{manastas@dm.uba.ar}

\address{\textrm{(C. Cabrelli)}
Departamento de Matem\'atica,
Facultad de Ciencias Exac\-tas y Naturales,
Universidad de Buenos Aires, Ciudad Universitaria, Pabell\'on I,
1428 Buenos Aires, Argentina and
CONICET, Consejo Nacional de Investigaciones
Cient\'ificas y T\'ecnicas, Argentina}
\email{cabrelli@dm.uba.ar}

 \thanks{The research of
M. Anastasio and C. Cabrelli  is partially supported by
Grants: PICT 15033, CONICET, PIP 5650, UBACyT X058 and X108}

\date{\today}
\maketitle

\begin{abstract}
A new paradigm in sampling theory has been developed recently by Lu and Do.
In this new approach the classical linear model is replaced by a non-linear, but structured
model consisting of a union of subspaces. This is the natural approach for the new theory of
compressed sampling, representation of sparse signals and signals with finite
rate of innovation. In this article we extend the theory of Lu and Do,
for the case that the subspaces in the union are shift-invariant spaces. We describe  the subspaces by means of frame generators instead of orthonormal bases.
We show that, the one to one and stability  conditions for  the sampling operator, are valid for this more general case.
\vspace{5mm} 
\vspace{3mm}
\end{abstract}

\section{Introduction}
Recently, Lu and Do \cite{DL08} extended the sampling problem assuming
that the signals to be sampled belong to a union of subspaces instead of a single subspace.
This approach represents a new paradigm in sampling theory.

In the classical setting the  signals are assumed to belong to a  single space of functions,
usually the Paley-Wiener space of band-limited functions. Since in many applications
the band-limitedness hypothesis is not realistic, other spaces of functions were considered,
mainly, shift-invariant spaces (SIS) with very ge\-neral generators.
In the approach of Lu and Do, the signals belong to a union of subspaces instead of a single one.
This simple idea  may have a great impact in many applications in signal processing,
in particular in the emerging theory of compressed sensing \cite{CT06}, \cite{CRT06}, \cite{Don06} and signals with finite rate of innovations \cite{VMB02}.

To describe the problem,  assume that $\F$ is a union of  subspaces from some Hilbert space
 $\mathcal{H}$ and a signal $s$ is extracted from $\F$. We take some measurements of that signal. These measurements can be thought of as the result of the  application of a series of functionals $\{\varphi_{\alpha}\}_{\alpha}$ to our signal $s.$
The problem is then to reconstruct the signal  using only the measurements $\{\varphi_{\alpha}(s)\}_{\alpha}$ and some des\-cription of the subspaces in $\F$.
The series of functionals define  an operator, {\it the sampling operator}, acting on the
ambient space $\mathcal{H}$ and taking values in a suitable sequence space.
Under some hypothesis on the structure of the subspaces, Lu and Do found necessary and sufficient conditions on these functionals in order for the sampling operator to be  stable and one-to-one when restricted to the union of the subspaces.
These conditions were obtained in two settings. In the euclidian space and in $L^2(\R^d)$.
In this latter case the subspaces considered were finitely generated shift-invariant spaces.

Blumensath and Davies \cite{BD07} studied the problem of sampling in union of subspaces  in the finite dimensional case, extending some of  the results in  Lu and Do \cite{DL08}.  They applied their results to  compressed sensing models and sparse signals.
In \cite{EM08}, Eldar developed a general framework for robust and efficient recovery of a signal from a given set of samples. The signal is a finite length vector that is sparse in some given basis and is assumed to lie in a union of subspaces.
Aldroubi et al in \cite{ACM08} established the existence of an optimal union of subspaces model
for a given data set in an abstract setting and considered the finite dimensional case and the shift-invariant case for $L^2(\R^d)$. They also developed an algorithm to find the model that fits the data set.

There are two technical aspects in the approach of Lu and Do that restrict the applicability  of their results in the shift-invariant space case.
The first one is due to the fact that the conditions are obtained in terms of Riesz bases
of translates of the SISs involved, and it is well known that not every SIS has a Riesz basis
of translates.
The second one is that the approach is based upon the sum of every two of the SISs in the union.
The conditions on the sampling operator are then obtained using fiberization  techniques on that sum.
This requires that the sum of each of two subspaces is a closed subspace, which is not true in general.

In this article we obtain the conditions for the sampling operator to be one-to-one and stable in terms
of {\it frames} of translates of the SISs instead of orthonormal basis. This extends the previous results to arbitrary SISs and in particular removes the restrictions mentioned above.
It is very important to have conditions based on frames, specially for applications, since frames are more flexible and simpler to construct. Frames of translates for shift-invariant spaces with generators that are smooth  and with good decay can be easily obtained.

On the other side we show that, using known results from the theory of SISs,  it is possible to determine families of subspaces on which the conditions for stability and injectivity are necessary and sufficient.

The article is organized in the following way:  Section 2 contains some notation and basic results that will be needed throughout. In Section 3 we set the problem of sampling in a union of subspaces in the general context of an abstract Hilbert space. We also give injectivity and stability conditions for the sampling operator, within this general setting. The case of finite-dimensional subspaces is studied in Section 4. In Section 5 we analyze the problem for   the Hilbert space  $L^2(\R^n)$ and  sampling in a union of finitely generated shift-invariant spaces.
Finally in Section 6  we use the notion of angle between subspaces to obtain necessary and sufficient conditions
for the closedness of the sum of two shift-invariant spaces.

\section{Preliminaries}

We will assume that $I$ and $J$ are countable index sets and $\mathcal{H}$ is a separable Hilbert space
over the complex field.

\begin{definition}\label{frames}

A sequence $X=\{x_j\}_{j\in J}$  in $\mathcal{H}$ is a {\it Bessel sequence} if there
exists a constant $0<\beta<+\infty$ such that
$$\sum_{j\in J}|\langle h,x_j\rangle|^2\leq\beta\,\|h\|^2_{\mathcal{H}}\quad\forall\,h\in \mathcal{H}.$$

If in addition there exist constants $0<\alpha\leq\beta<+\infty$ such that

$$\alpha\,\|h\|^2_{\mathcal{H}}\leq\sum_{j\in J}|\langle h,x_j\rangle|^2\leq\beta\,\|h\|^2_{\mathcal{H}}\quad\forall\,h\in \mathcal{H},$$
then $X$ is said to be a {\it frame} for $\mathcal{H}$. The sequence X  is a {\it Parseval frame} if $\alpha$ and $\beta$ can be chosen so that $\alpha=\beta=1.$

The closure of the span of $X$ will be denoted by $\sn \{x_j\}_{j\in J}$. $X$ is a
{\it frame sequence} if it is a frame for $\sn \{x_j\}_{j\in J}$.

\end{definition}

\begin{definition}

A sequence $X=\{x_j\}_{j\in J}$ in $\mathcal{H}$ is a {\it Riesz basis} for $\mathcal{H}$ if it is complete in $\mathcal{H}$
and there exist constants $0<\alpha\leq\beta<+\infty$ such that

$$\alpha\,\sum_{j\in J}|c_j|^2\leq\Big\|\sum_{j\in J}c_j x_j\Big\|^2\leq\beta\,\sum_{j\in J}|c_j|^2\quad\forall\,\{c_j\}_{j\in J}\in \ell^2(J).$$

\end{definition}

\begin{definition}
If $X=\{x_j\}_{j\in J}$ is a Bessel sequence in $\mathcal{H}$, we define
the \textit{analysis operator} as
$$B_X:\mathcal{H}\rightarrow \ell^2(J),\quad B_Xh=\{\langle
h,x_j\rangle\}_{j\in J}.$$
The adjoint of $B$ is the \textit{synthesis operator}, given by
$$B_X^*:\ell^2(J)\rightarrow \mathcal{H},\quad B_X^*c=\sum_{j\in J}
c_j x_j.$$
The Bessel condition guarantees the boundedness of $B_X$ and as a consequence, that of $B_X^*$.
\end{definition}

\begin{definition}\label{Gramian}
Suppose $X=\{x_j\}_{j\in J}$ is a Bessel sequence in $\mathcal{H}$ and $B_X$ is the analysis operator.
The Gramian of the system $X$ is defined by
$$G_X:\ell^2(J)\rightarrow \ell^2(J),\quad G_X:=B_XB_X^*.$$
We identify  $G_{X}$ with its matrix representation.
$$(G_{X})_{j,k}=\langle x_k,x_j\rangle\quad\forall\,j,k\in J.$$
\end{definition}

Given a Hilbert space $\mathcal{L}$  and a bounded linear operator $W:\mathcal{L}\rightarrow\mathcal{L}$ , we will denote by $\sigma(W)$ the spectrum of $W$.

The following is a well known property. Its proof can be deduced from \cite[Lemma 5.5.4]{Chr03},

\begin{proposition}\label{fra}

Let $X:=\{x_j\}_{j\in J}\subset \mathcal{H}$ be a Bessel sequence, then $X$ is a frame sequence with constants $\alpha$ and $\beta$ if and only if
$$\sigma (G_X)\subset \{0\}\cup[\alpha,\beta].$$
\end{proposition}

\begin{definition}
Let $X:=\{x_j\}_{j\in J}$ and $Y:=\{y_i\}_{i\in I}$ be Bessel sequences in $\mathcal{H}$. Let  $B_{X}$ and $B_{Y}$ be the analysis operators associated to $X$ and $Y$ respectively. The {\it cross-correlation} operator is defined by
\begin{equation}\label{Gcross}
G_{X,Y}:\ell^2(J)\rightarrow \ell^2(I),\quad G_{X,Y}:= B_{Y}B_{X}^*.
\end{equation}
Identifying  again  $G_{X,Y}$ with its matrix representation, we write
$$(G_{X,Y})_{i,j}=\langle x_j,y_i\rangle\quad\forall\, j\in J,\forall\, i \in I.$$
\end{definition}

We will need the following property of Parseval frames.

\begin{proposition}\label{obpro1}

If $X =\{x_j\}_{j\in J}$ is a Parseval frame for a closed subspace $S$, and $B_X$ is the analysis operator associated to $X$,
then the orthogonal projection of $\mathcal{H}$ onto $S$ is
$$P_S=B_X^*B_X:\mathcal{H}\rightarrow\mathcal{H}, \quad P_Sh=\sum_{j\in J}\langle h, x_j\rangle\, x_j.$$
\end{proposition}

\begin{remark}
A Parseval frame for a Hilbert space does not need to be an orthogonal system. In fact, it is orthogonal if and only if every element of the set has unitary norm.
A simple example is the family  $X=\{\frac{1}{\sqrt{2}}e_1, \frac{1}{\sqrt{2}}e_1,e_n\}_{n\geq 2}$
where  $\{e_n\}_{n\in\N}$ is an orthonormal basis for the Hilbert space.
X is a Parseval frame that is not orthogonal and it is not even a basis.
\end{remark}

\section{The sampling operator}\label{section3}

Let $\mathcal{H}$ be a separable Hilbert space and $S \subset\mathcal{H}$ an arbitrary set.
Given $\Psi=\{\psi_i\}_{i\in I}$ a Bessel sequence in $\mathcal{H}$, the sampling problem consists of reconstructing
a signal $f\in S$ using the data $\{\langle f,\psi_i\rangle\}_{i\in I}$. We first require that the signals are uniquely determined by the data. That is, if we define the
{\it Sampling operator} by
\begin{equation}\label{op}
A:\mathcal{H}\rightarrow \ell^2(I), \quad Af:=\{\langle f,\psi_i\rangle\}_{i\in I},
\end{equation}
we require $A$ to be one-to-one on $S$. The set $\Psi$ will be called the {\it Sampling set}.

Note that the sampling operator $A$ is the analysis operator for the sequence  $\Psi$.

Another important property that is usually required for a sampling operator, is stability.
This is crucial to bound the error of  reconstruction in noisy situations.

The stable sampling condition was first proposed by \cite{Lan67} for the case when $S$ is the Paley-Wiener space.
It was then generalized in \cite{DL08} to the case when $S$ is a union of subspaces.

\begin{definition}
A sampling operator  $A$ is called {\it stable} on $S$ if there exist two constants
$0<\alpha\leq\beta<+\infty$ such that
$$\alpha\|x_1-x_2\|^2_{\mathcal{H}}\leq\|Ax_1-Ax_2\|^2_{\ell^2(I)}\leq\beta\|x_1-x_2\|^2_{\mathcal{H}}\quad\forall\,x_1,x_2\in S.$$
\end{definition}

When $S$ is a closed subspace, the injectivity and the stability can be expressed in terms of conditions on $P_S\Psi$, where $P_S$ is the orthogonal
projection of $\mathcal{H}$ onto $S$.

\begin{proposition}\label{reiny}
Let $\H$ be a Hilbert space,  $S \subset \H$  a closed subspace and $\Psi=\{\psi_i\}_{i\in I}$ a Bessel sequence in $\mathcal{H}$. If $A$ is the sampling operator associated to $\Psi,$ then we have
\begin{enumerate}
\item[i)] \label{one} The  operator $A$ is {\em one-to-one} on $S$ if and only if
$\{P_S\psi_i\}_{i\in I}$ is complete in S, that is $S=\overline{\textnormal{span}}\{P_S\psi_i\}_{i\in I}$.
\item [ii)]  The operator A is {\em stable} on $S$ with constants $\f$ and $\beta$ if and only if
$\{P_S\psi_i\}_{i\in I}$
is a frame for $S$ with constants $\f$ and $\beta$.
\end{enumerate}
\end{proposition}

\begin{proof}

The proof of {\it i)} is straightforward using that if $f\in S$ then
$$\langle f, P_S\psi_i\rangle=\langle P_S f,\psi_i\rangle=\langle f,\psi_i\rangle.$$
 For {\it ii)} note that for all  $f \in S$
$$\|Af\|^2_{\ell^2(I)}=\sum_{i\in I}|\langle f, \psi_n\rangle|^2=\sum_{i\in I}|\langle P_S f,\psi_n\rangle|^2=
\sum_{i\in I}|\langle f,P_S\psi_n\rangle|^2.$$
\end{proof}

\begin{remark}
Given a closed subspace $S$ in a Hilbert space $\H$, a sequence of vectors $\{\psi_i\}_{i\in I} \subset \H$
is called an {\it outer frame } for $S$ if $\{P_S\psi_i\}_{i \in I}$ is a frame for $S$. The notion of
outer frame was introduced in \cite{ACM04}. See also \cite{FW01} and \cite{LO04} for  related definitions.
Using this terminology, part ii) of Proposition \ref{reiny} says that the sampling operator $A$ is stable if
and only if  $\{\psi_i\}$ is an outer frame for $S$.
\end{remark}

In what follows we will extend one-to-one and stability conditions for the operator $A$,
to the case of a union of subspaces instead of a single subspace.

If $\{S_{\gamma}\}_{\gamma\in\Gamma}$ are closed subspaces of $\mathcal{H}$, with $\Gamma$ an arbitrary index set.
Let $$\chi:=\bigcup_{\gamma\in\Gamma} S_{\gamma}.$$

We want to study conditions on $\Psi$ so that the sampling operator $A$ defined by (\ref{op})
is one-to-one and stable on $\chi$.

This study continues the one initiated by Lu and Do \cite{DL08} in which they translated the conditions on $\chi$ into conditions on the subspaces defined by
\begin{equation}\label{suma}
S_{\g,\ti}:=S_{\gamma}+S_{\theta}=\{x+y\,:\,x\in S_{\gamma},\,y\in S_{\theta}\}.
\end{equation}

Working with the subspaces $S_{\g,\ti}$ instead of $\chi$, allows to exploit lineal pro\-perties of $A$.

They proved the following proposition.

\begin{proposition}\cite{DL08}\label{proposition-dl} With the above notation we have,
\begin{enumerate}
\item[i)]
The operator A is one-to-one on $\chi$ if and only if A is one-to-one on every  $S_{\g,\ti}$
with $\g,\ti\in\Gamma.$
\item[ii)]
The operator A is stable for $\chi$ with stability bounds $\alpha$ and $\beta$, if and only if A is stable for $S_{\g,\ti}$ with stability bounds $\alpha$ and $\beta$
for all $\g,\ti\in\Gamma,$ i.e.
$$\alpha\|x\|^2_{\mathcal{H}}\leq\|Ax\|^2_{\ell^2(I)}\leq\beta\|x\|^2_{\mathcal{H}}\quad\forall\,x\in S_{\g,\ti},\forall\ \g,\ti\in\Gamma.$$
\end{enumerate}
\end{proposition}

The sum of two closed infinite-dimensional  subspaces of a Hilbert space is not necessarily
 closed (see Example \ref{ejclau}). Furthermore, the injectivity of an operator on a subspace does not imply the injectivity on its closure. So,
we can not apply Proposition \ref{reiny} to the subspaces $S_{\g,\ti}.$
However, we can obtain a sufficient condition for the injectivity.
\begin{proposition}\label{complete}
 If $\{P_{\overline{S}_{\g,\ti}}\psi_i\}_{i\in I}$ is complete on $\overline{S}_{\g,\ti}$ for every $\g,\ti\in\Gamma,$ then $A$ is one-to-one on $\chi.$
\end{proposition}

When the  subspaces of the family $\{{S}_{\g,\ti}\}_{\g,\ti\in\Gamma}$ are all  closed,
the condition in Proposition \ref{complete}, will be also necessary for
the injectivity of $A$ on $\chi$.
So, a natural question will be, when the sum of two closed subspaces  of a Hilbert space is  closed.
In Section \ref{section6} we study this problem in several situations.

In the case of the stability, Proposition \ref{reiny} can be applied since, by the boundedness of $A$, we have the following.

\begin{proposition}\label{lemstab}
Let $S$ be a subspace of $\mathcal{H}$, the operator A is stable for $S$  with constants $\alpha$ and $\beta$ if and only if it is stable for
$\overline{S}$ with constants $\alpha$ and $\beta$.
\end{proposition}

As a consequence of this, using Propositions \ref{reiny} and  part ii) of Proposition \ref{proposition-dl}, we
have

\begin{proposition}
$A$ is stable for $\chi$ with constants
$\alpha$ and $\beta$ if and only if
$\{P_{\overline{S}_{\g,\ti}}\psi_i\}_{i\in I}$ is a frame for $\overline{S}_{\g,\ti}$ for every $\g,\ti\in\Gamma$ with the same constants $\alpha$ and $\beta$.
\end{proposition}

\section{Union of finite-dimensional subspaces}\label{section4}

In this section  we will first obtain conditions on the sequence $\{\psi_i\}_{i\in I}$ for the sampling operator to be  one-to-one on a union of finite-dimensional subspaces. We will then analyze the stability requirements.
We are interested in expressing these conditions in terms of the generators of the sum of every two subspaces of the union.

\subsection{The one-to-one condition for the sampling operator}

Let $\mathcal{H}$ be a Hilbert space, $\Psi=\{\psi_i\}_{i\in I}$  a Bessel sequence in $\mathcal{H}$,  and $A$ the sampling operator associated to $\Psi$ as in (\ref{op}).

Let $S$ be a finite-dimensional subspace of $\H$ and  $\Phi=\{\phi_j\}_{j=1}^{d}$ a finite frame for $S$.
(Recall that a finite set of vectors from a finite-dimensional subspace is a frame for that subspace if and only if it spans it.)

The cross-correlation operator associated to $\Psi$ and $\Phi$ (see (\ref{Gcross})) in this case can be written as,
$$G_{\Phi,\Psi}:\C^d\rightarrow \ell^2(I),\quad G_{\Phi,\Psi}=AB_{\Phi}^*,$$
where $B^*_{\Phi}:\C^d\rightarrow \mathcal{H}$ is the synthesis operator associated to $\Phi$.

The next theorem gives necessary and sufficient  conditions on the cross-correlation operator for the sampling operator to be one-to-one on $S$.

\begin{theorem}\label{teinys}
Let $\Psi=\{\psi_i\}_{i\in I}$ be a Bessel sequence for $\mathcal{H}$, $S$ a finite-dimensional  subspace of $\mathcal{H}$ and $\Phi=\{\phi_j\}_{j=1}^{d}$ a frame for $S.$
Then the following are equivalent:
\begin{enumerate}
\item [i)] $\Psi$ provides a one-to-one sampling operator on $S.$
\item [ii)] $\ker(G_{\Phi,\Psi})=\ker(B^*_{\Phi}).$
\item [iii)] $\dim(\textnormal{range}(G_{\Phi,\Psi}))=\dim(S).$
\end{enumerate}
\end{theorem}

\begin{proof}
The proof is straightforward using that the range of the operator $B^*_{\Phi}$ is $S$.
\end{proof}

\begin{remark}
Note that the conditions in Theorem \ref{teinys} do not depend on the particular  chosen frame.
That is, if there exists a frame $\Phi$ for $S$, such that
 $\dim(\text{range}(G_{\Phi,\Psi}))=\dim(S),$ then
  $\dim(\text{range}(G_{\widetilde{\Phi},\Psi}))=\dim(S),$ for any
 frame $\widetilde{\Phi}$ for $S$.
\end{remark}

Now  we will apply the previous theorem for the case of a union of subspaces.

Let $\{S_{\gamma}\}_{\gamma\in\Gamma}$ be a collection of finite-dimensional subspaces of $\mathcal{H}$, with $\Gamma$ an arbitrary index set.
Define, $$\chi:=\bigcup_{\gamma\in\Gamma} S_{\gamma}.$$
As before, set $S_{\g,\ti}:=S_{\gamma}+S_{\theta}.$

We obtain the following result which extends the result in \cite{DL08} to the case that the subspaces in the union are described by frames.

\begin{theorem}\label{teiny}
Let $\Psi=\{\psi_i\}_{i\in I}$ be a Bessel sequence for $\mathcal{H}$ and for every $\g,\ti\in\Gamma$, let $\Phi_{\g,\ti}$ be a frame for $S_{\g,\ti}$,
the following are equivalent:
\begin{enumerate}
\item [i)] $\Psi$ provides a one-to-one sampling operator on $\chi.$
\item [ii)] $\dim(\textnormal{range}(G_{\Phi_{\g,\ti},\Psi}))=\dim(S_{\g,\ti})$ for all $\g,\ti\in\Gamma.$
\end{enumerate}
\end{theorem}

Note that if $I$ is a finite set, the problem of testing the injectivity of $A$ on $\chi$ reduces to check that the rank of the cross-correlation
matrices are equal to the dimension of the subspaces $S_{\g,\ti}$.

In this case a lower bound for the cardinality of the sampling set can be established. This is stated
in the following corollary from \cite{DL08}. We include a proof of the result based on Theorem \ref{teiny}.

Here  $\# I$  denotes the cardinality of the finite set $I$.

\begin{corollary}
If the operator A is one-to-one on $\chi$ and $I$ is finite, then
$$\# I \geq \sup_{\g,\ti\in \Gamma}(\dim(S_{\g,\ti})).$$
\end{corollary}

\begin{proof}
Since $I$ is finite, we have that $\text{range} (G_{\Phi_{\g,\ti},\Psi})\subset\C^{\#I}$. Thus, using
part ii) of Theorem \ref{teiny}, we obtain that
$$\dim(S_{\g,\ti})=\dim(\text{range} (G_{\Phi_{\g,\ti},\Psi}))\leq \# I,\quad\forall\,\g,\ti\in\Gamma.$$
\end{proof}

\subsection{The stability condition for the sampling operator}

We are now interested in studying  conditions for stability of the sampling ope\-rator.
These conditions will be set in terms of the  cross-correlation operator.
We will consider Parseval frames to obtain simpler conditions.

Given Hilbert spaces $\mathcal{L}$ and $\mathcal{L'}$ and a bounded linear operator $W:\mathcal{L}\rightarrow\mathcal{L'}$, we denote by $\sigma^2(W)$  the set
$$\sigma^2(W)=\sigma(W^*W).$$

\begin{theorem}\label{teosta}
Let $\Psi=\{\psi_i\}_{i\in I}$ be a Bessel sequence for $\mathcal{H}$, $S$ a finite-dimensional  subspace of $\mathcal{H}$ and $\Phi$ a Parseval frame for $S$.

The sequence $\Psi$ provides a stable sampling operator for $S$  with constants $\alpha$ and $\beta$ if and only if
\begin{enumerate}
\item [i)] $\dim(\textnormal{range}(G_{\Phi,\Psi}))=\dim(S)$ and
\item [ii)] $\sigma^2(G_{\Phi,\Psi})\subseteq\{0\}\cup[\alpha,\beta].$
\end{enumerate}
\end{theorem}

\begin{proof}

Let $W:\mathcal{H}\rightarrow \ell^2(I),$ be  the analysis operator associated to
$P_S\Psi$.
For $x \in \H$, the equation,
$$Wx=\{\langle x,P_{S}\psi_i\rangle\}_{i\in I}=
\{\langle P_{S}x,\psi_i\rangle\}_{i\in I}=AP_{S}x,$$
shows that $W=AP_{S}.$

Since $\Phi$ is a Parseval frame for $S$, by Proposition \ref{obpro1}, $P_S = B_{\Phi}^*B_{\Phi}^{}$ then,
\begin{equation}\label{cuenta}
G_{P_S\Psi}=WW^*=
AP_{S}P_{S}A^*
=AP_{S}A^*=AB_{\Phi}^*B_{\Phi}^{}A^*=G_{\Phi,\Psi}G_{\Phi,\Psi}^*.
\end{equation}

Let us assume first that $A$ is stable for $S$. Item i) follows from Theorem \ref{teinys}.
Now we prove ii).

Since  $ S$ is closed ($S$ is finite dimensional) then Proposition \ref{reiny} gives that
$P_S\Psi:=\{P_{S}\psi_i\}_{i\in I}$
is a frame for $S$ with constants $\f$ and $\beta$. Using Proposition \ref{fra}, we have,
$$\sigma (G_{P_S\Psi})\subset \{0\}\cup[\alpha,\beta].$$

So, by (\ref{cuenta}),
$$\sigma (G_{P_S\Psi})=\sigma (G_{\Phi,\Psi}G_{\Phi,\Psi}^*)\subset \{0\}\cup[\alpha,\beta].$$
Finally, since (see  \cite{Rud91})
$$\sigma (G_{\Phi,\Psi}^*G_{\Phi,\Psi})\subset \{0\}\cup\sigma (G_{\Phi,\Psi}G_{\Phi,\Psi}^*),$$
it follows that
$$
\sigma ^2(G_{\Phi,\Psi})\subset\{0\}\cup[\alpha,\beta].$$

Suppose now that i) and ii) hold. Recall that A is stable for $S$ with stability bounds $\alpha,\beta$ if and only if $P_S\Psi:=\{P_{S}\psi_i\}_{i\in I}$ is
a frame for $S$ with frame bounds $\alpha,\beta$.

By Theorem \ref{teinys}, condition i) implies that the sampling operator is one-to-one on $S$.
Thus, using Proposition \ref{reiny},
$P_S\Psi:=\{P_{S}\psi_i\}_{i\in I}$ is complete in $S.$

That $P_S\Psi:=\{P_{S}\psi_i\}_{i\in I}$  is a frame sequence is straightforward by ii),
(\ref{cuenta}) and Proposition \ref{fra}.
\end{proof}

\begin{remark}
As in the case of injectivity, we note that the condition of stability does not depend on the  chosen Parseval frame. That means, if condition i) and ii) in the previous theorem hold for a Parseval frame $\Phi$ for $S$, then they hold for any Parseval frame $\widetilde{\Phi}$ for $S$.
\end{remark}

Theorem \ref{teosta} applied to the union of subspaces gives:

\begin{theorem}\label{teo}
Let $\Psi=\{\psi_i\}_{i\in I}$ be a set of sampling vectors and for every $\g,\ti\in\Gamma$, let $\Phi_{\g,\ti}$ be a Parseval frame for $S_{\g,\ti}$.

The sequence $\Psi$ provides a stable sampling operator for $\chi$  with constants $\alpha$ and $\beta$ if and only if
\begin{enumerate}
\item [i)] $\dim(\textnormal{range}(G_{\Phi_{\g,\ti},\Psi}))=\dim(S_{\g,\ti})$ for all $\g,\ti\in \Gamma$ and
\item [ii)] $\sigma^2(G_{\Phi_{\g,\ti},\Psi})\subseteq\{0\}\cup[\alpha,\beta]$  for all $\g,\ti\in\Gamma.$
\end{enumerate}
\end{theorem}

For examples and existence of sequences  $\Psi$ which verify the conditions of injectivity or stability in a union of finite-dimensional subspaces, we refer the reader to \cite{BD07} and \cite{DL08}.

\section{Sampling in a union of finitely generated shift-invariant spaces}

In this section we will consider the case of the Hilbert space   $\mathcal{H}=L^2(\R^n)$ and
finitely generated shift-invariant spaces (FSISs). That is, we will study sampling in a union of FSISs.

We will first review some properties of these spaces. For a detailed treatment see  \cite{BDVR94,Bow00, Hel64, RS95} and the references therein.

\subsection{Some basic facts about shift-invariant spaces }

We use the Fourier transform defined by $$\hat{f}(\w)=\int_{\R^n}f(x)\,e^{-2\pi i\w x}\,dx$$
for $f\in L^1(\R^n)$, and extended to be a unitary operator on $L^2(\R^n)$.

$\mathbb{T}^n=\R^n/\Z^n$ is the torus group, identified with $[0,1)^n$.

The translation by $k\in\Z^n$ is denoted by $t_k f:=f(\cdot-k).$

\begin{definition}
A closed subspace $S\subset L^2(\R^n)$ is a {\it shift-invariant space} (SIS) if $f\in S$ implies $t_k f\in S$ for any $k\in\Z^n.$

Given $\Phi$ a set of functions in $L^2(\R^n)$, we denote by $E(\Phi)$ the set,
$$E(\Phi):=\{t_k\phi:  k\in\Z^n, \phi \in \Phi \}.$$

The SIS generated by this set is $$S(\Phi):=\sn(E(\Phi)).$$
If $S=S(\Phi)$ for some finite set $\Phi$ we say that $S$ is a \textit{finitely generated shift-invariant space} (FSIS).

If $S$ is an FSIS, we call  the {\it length} of S, the cardinality of a smallest generating
set for S, and write
$$\mbox{len}(S):=\min\{\#\Phi\,:\,S=S(\Phi)\}.$$
\end{definition}

Although the FSISs are infinite-dimensional subspaces, most of their pro\-perties can be translated into properties on the
fibers of the spanning sets. That allows to work with finite-dimensional subspaces of  $\ell^2(\Z^n)$. We will
give the definition and some properties of the fibers.

The Hilbert space of square integrable vector functions $L^2(\mathbb{T}^n,\ell^2(\Z^n))$,
consists of all vector valued measurable functions $F:\mathbb{T}^n\rightarrow \ell^2(\Z^n)$ such that
$$\|F\|:=\Big(\int_{\mathbb{T}^n}\|F(x)\|^2_{\ell^2}\,dx\Big)^{\frac{1}{2}},$$is finite.

\begin{proposition}
The function $\tau:L^2(\R^n)\rightarrow L^2(\mathbb{T}^n,\ell^2(\Z^n))$ defined for $f\in L^2(\R^n)$ by
$$\tau f(\w):=\{\hat{f}(\w+k)\}_{k\in\Z^n},$$
is an isometric isomorphism between  $L^2(\R^n)$ and $L^2(\mathbb{T}^n,\ell^2(\Z^n))$.

The sequence  $\{\hat{f}(\w+k)\}_{k\in\Z^n}$  is called the {\it fiber} of $f$ at $\w$.
\end{proposition}

\begin{definition}\label{def-range}
A \textit{range function} is a mapping
$$J:\mathbb{T}^n\rightarrow\{\text{closed subspaces of } \ell^2(\Z^n)\}.$$
$J$ is measurable if the operator valued function of the orthogonal projections $\w\mapsto P_{J(\w)}$ is weakly measurable.

Note that in a separable Hilbert space measurability is equivalent to weak measurability. Therefore, the measurability of $J$
is equivalent to $\w\mapsto P_{J(\w)}(a)$ being vector measurable for each $a\in\ell^2(\Z^n)$, or
$\w\mapsto P_{J(\w)}(F(\w))$ being vector measurable for each fixed vector measurable function
$F:\mathbb{T}^n\rightarrow \ell^2(\Z^n)$.
\end{definition}

\begin{proposition}\label{fun}
A closed subspace $S\subset L^2(\R^n)$ is shift-invariant if and only if
$$
S=\{f\in L^2(\R^n)\,:\,\tau f (\w)\in J_S(\w)\text{ for a.e. }\w\in\mathbb{T}^n\},
$$
where $J_S$ is a measurable range function. The correspondence between $S$ and $J_S$ is one-to-one.

Moreover, if $S=S(\Phi)$ for some countable set $\Phi\subset L^2(\R^n)$, then
$$J_S(\w)=\overline{\textnormal{span}}\{\tau\phi (\w)\,:\,\phi\in\Phi\}\quad\text{for a.e. }\w\in\mathbb{T}^n.$$ The subspace $J_S(\w)$ is called
 the \textit{fiber space} of $S$ at $\w$.
\end{proposition}

\begin{proposition}\label{proy}
Let $S$ be a SIS of $L^2(\R^n)$ and $f\in L^2(\R^n),$ then
$$\tau(P_S f)(\w)=P_{J_S(\w)}(\tau f (\w))\quad\mbox{ for a.e. }\w\in\mathbb{T}^n.$$
\end{proposition}

\begin{definition}
Given $S$ a SIS of $L^2(\R^n)$, the  {\it dimension function} associated to $S$ is defined by
$$\dim_S:\mathbb{T}^n\rightarrow\N_0\cup\{\infty\},\quad\dim_S(\w)=\dim(J_S(\w)).$$
\end{definition}
Here $\N_0$ denotes the set of non-negative integers.

The next two theorems characterize Bessel sequences, frames and Riesz bases of translates in terms of fibers.
\begin{theorem}\label{grami5}
Let $\Phi$ be a countable subset of $L^2(\R^n)$. The following are equi\-valent.
\begin{enumerate}
\item[i)] $E(\Phi)$ is a Bessel sequence in $L^2(\R^n)$ with constant $\beta.$
\item[ii)] $\tau\Phi (\w):=\{\tau\phi (\w)\,:\,\phi\in\Phi\}$ is a Bessel sequence in $\ell^2(\Z^n)$ with constant $\beta$ for a.e. $\w\in\mathbb{T}^n.$
\end{enumerate}
\end{theorem}

Let $\Phi = \{\phi_j\}_{j\in J}$ be a countable set of functions in $L^2(\R^n)$ such that $E(\Phi)$ is a Bessel sequence. The Gramian of $\Phi$ at $\w\in\mathbb{T}^n$ is
$$\G_{\Phi}(\w):\ell^2(J)\rightarrow \ell^2(J), \quad
(\G_{\Phi}(\w))_{i,j}=\langle \tau \phi_j(\w),\tau\phi_i(\w)\rangle_{\ell^2(\Z^n)}\quad\forall\,i,j\in J.$$
Note that $\G_{\Phi}(\w)$ is the Gramian operator associated to  the Bessel sequence $\{\tau\phi_j(\w)\}_{j\in J}$ in $\ell^2(\Z^n)$, that is  $\G_{\Phi}(\w)= G_{\tau\Phi(\w)},$ (see Definition \ref{Gramian}).

\begin{theorem}\label{grami3}
Let $S=S(\Phi)$, where $\Phi$ is a countable subset of $L^2(\R^n)$. Then the following holds:
\begin{enumerate}
\item[i)] $E(\Phi)$ is a frame for $S$ with constants $\alpha$ and $\beta$ if and only if
$\tau\Phi(\w)$ is a frame for $J_S(\w)$ with constants $\alpha$ and $\beta$ for a.e. $\w\in\mathbb{T}^n.$
\item[ii)] $E(\Phi)$ is a Riesz basis for $S$ with constants $\alpha$ and $\beta$ if and only if
$\tau\Phi(\w)$ is a Riesz basis for $J_S(\w)$ with constants $\alpha$ and $\beta$ for a.e. $\w\in\mathbb{T}^n.$

Furthermore, if $\Phi$ is finite, $S$ has a Riesz basis of translates if and only if the dimension
function associated to $S$ is constant a.e. $\w\in\mathbb{T}^n.$
\end{enumerate}
\end{theorem}

Let us remark here that if   $\Phi\subset L^2(\R^n)$ is a set of generators for a shift-invariant space $S$, that is $S=S(\Phi)$,
then the set $E(\Phi)$ does not need to be a frame for $S$, even for finitely generated SISs. However it is always true that there exists a set  of generators for $S$ such that its integer translates form a frame for $S.$ This is the result of the next theorem.

\begin{theorem}\label{RS}
Given $S$ a SIS of $L^2(\R^n)$, there exists a  subset $\Phi=\{\phi_j\}_{j\in J}\subset S$ such that
$E(\Phi)$ is a Parseval frame for $S$.
If S is finitely generated, the cardinal of $J$  can be chosen to be the length of $S.$
\end{theorem}

Although a SIS always has a frame of translates, there are SISs which have no Riesz bases of translates.
For example, consider the shift-invariant space $S$ generated by $\phi\in L^2(\R)$, where $\hat{\phi}(\w)=\chi_{[0,\frac{1}{2})}(\w).$
Since  $\dim_S(\w)=1$ for a.e. $\w\in[0,\frac{1}{2})$ and $\dim_S(\w)=0$ for a.e. $\w\in[\frac{1}{2},1)$, it follows by Theorem \ref{grami3} that $S$ has no Riesz bases of translates.

\subsection{Sampling from a Union of FSIS}

In this section we will study the sampling problem for the case in which the signal belongs to the set,
\begin{equation}\label{unionSIS}
\chi:=\bigcup_{\gamma\in\Gamma}S_{\gamma},
\end{equation}
where $S_{\gamma}$ are FSISs of $L^2(\R^n)$.

In this setting, since our subspaces are shift-invariant, it is natural and also convenient that the sampling set will be the set of shifts from a fixed collection of functions in  $L^2(\R^n)$, that is, the sampling operator will be given by a sequence
of integer translates of certain functions.

Given $\Psi:=\{\psi_i\}_{i\in I}$ such that   $E(\Psi)$ is a Bessel sequence in $L^2(\R^n)$, we define the sampling operator associated to $E(\Psi)$ as
\begin{equation}\label{A}
A:L^2(\R^n)\rightarrow \ell^2(\Z^n\times I),\quad
Af=\{\langle f, t_k\psi_i\rangle\}_{i\in I, k\in\Z^n}.
\end{equation}
As we showed in Section \ref{section3} the conditions on  the sampling operator to be  one-to-one and stable in a union of subspaces can be established in terms  of one-to-one  and stability conditions on the sum of every two of the subspaces from the union.

However the condition that we have for the sampling operator to be one-to-one on a subspace, requires that the subspace is closed (Proposition \ref{reiny}).

Since the sum of two FSISs  is  not necessarily a closed subspace, the conditions should be imposed on the closure of the sum.

Conditions that guarantee  that the sum of two FSISs is closed are described in Section \ref{section6}.

In what follows we will consider, for each $\g,\ti \in \Gamma$, the subspaces,
\begin{equation}\label{defclau}
\overline{S}_{\g,\ti}:=\overline{S_{\g}+S_{\ti}}.
\end{equation}
The following proposition states that the closure of the sum of two SISs is a SIS generated by the union
of the generators of the two spaces. Its proof is straightforward.

\begin{proposition}\label{clau}
Let $\Phi$ and $\Phi'$ be sets in $L^2(\R^n)$,
then $$\overline{S(\Phi)+S(\Phi') }=S(\Phi \cup\Phi').$$
In particular, if $S$ and $S'$are FSISs, then $\overline{S+S'}$ is an FSIS and
$$\textnormal{len}(\overline{S+S'}) \leq \textnormal{len}(S)+\textnormal{len}(S').$$
\end{proposition}

Now, as a consequence of Proposition \ref{clau}, for each $\g,\ti\in\Gamma$, $\overline{S}_{\g,\ti}$ is an FSIS.
Then, by Theorem \ref{RS}, we can choose, for each $\g,\ti\in\Gamma$, a {\it finite} set
$$\Phi_{\g,\ti}=\{\phi_j^{\g,\ti}\}_{ j=1}^{d_{\g,\ti}} $$
of $ L^2(\R^n)$ functions such that,
$$\overline{S}_{\g,\ti}=S(\Phi_{\g,\ti}),$$ and
$E(\Phi_{\g,\ti})$ forms a Parseval frame for $\overline{S}_{\g,\ti}$.

\subsection{The one-to-one condition}

We now study the conditions that the sampling set must satisfy  in order  for the operator $A$ defined by (\ref{A}) to be one-to-one on $\chi$.

Given a shift-invariant space $S$, the orthogonal projection onto $S$, denoted by $P_S$, commutes
with integer translates. Then, part i) of  Proposition \ref{reiny} can be rewritten as,

\begin{proposition}\label{reb}
Given a shift-invariant space $S$, $\Psi=\{\psi_i\}_{i\in I}$  such that $E(\Psi)$ is a Bessel sequence in $L^2(\R^n)$ and $A$ the  sampling operator associated to $E(\Psi)$. Then the following are equivalent.
\begin{enumerate}
\item [i)] The sampling operator $A$  is one-to-one on  $S.$
\item [ii)] $E( P_S\Psi)=\{t_k P_S\psi_i\}_{i\in I,k\in\Z^n}$
is complete in $S$, that is $S=\overline{\textnormal{span}} \,E( P_S\Psi)$.
\end{enumerate}
\end{proposition}

Since $E(\Psi)$  is a Bessel sequence in $L^2(\R^n)$, by Theorem \ref{grami5} we have that
$\{\tau\psi_i(\w)\}_{i\in I}$ is a Bessel sequence in $\ell^2(\Z^n)$ for a.e $\w\in\mathbb{T}^n$, so we can define (up to a set of measure zero), for $\omega \in \mathbb{T}^n$, the sampling operator related to the fibers:
$$\A(\w):\ell^2(\Z^n)\rightarrow \ell^2(I),$$ with
\begin{equation}\label{aw}
\A(\w)(c)=\{\langle c, \tau \psi_i(\w)\rangle\}_{i\in I}.
\end{equation}

That is, for a fixed $\omega \in \mathbb{T}^n$, we consider the problem of sampling from a union of subspaces in a different setting. The Hilbert space is $\ell^2(\Z^n)$,  the sequences of the sampling set are $\{\tau\psi_i(\w)\}_{i\in I}$, and the subspaces in the union are $J_{S_{\g}}(\omega),  \g \in \Gamma.$

Since the subspaces $\overline{S}_{\g,\ti}$ are FSISs, the fiber spaces $J_{\overline{S}_{\g\,\ti}}(\omega)$ are finite-dimensional.
So, the results of Section \ref{section4} can be applied, and conditions on the fibers can be obtained
in order for the operator $\A(\omega)$ to be one-to-one.

We are now going to show that given a finitely generated shift-invariant space $S$, the operator $A$ is one-to-one on $S$ if and
only if for almost every $\omega \in \mathbb{T}^n$,  the operator $\A(\w)$ is one-to-one on the corresponding fiber spaces $J_{S}(\omega)$ associated to $S$.
Once this is accomplished, we can apply the known conditions for the operator $\A(\omega).$

Given $\{t_k\phi_j\}_{j=1,k\in\Z^n}^{d}$ a Bessel sequence in $L^2(\R^n)$, we have the synthesis operator related to the fibers,
that is
\begin{equation}\label{b*w}
\B_{\Phi}^*(\w):\C^{d}\rightarrow
\ell^2(\Z^n),\quad
\B_{\Phi}^*(\w)(c_1,\ldots,c_d)=\sum_{j=1}^{d}c_j\tau\phi_j(\w).
\end{equation}
Note  that $\B_{\Phi}^*(\w)$ is the synthesis operator associated to the set $\tau\Phi(\w)$, that is $\B_{\Phi}^*(\w)=B_{\tau\Phi(\w)}^*$.

And we will have the cross-correlation operator associated to the fibers
$$\G_{\Phi,\Psi}(\w):\C^d\rightarrow \ell^2(I),\quad \G_{\Phi,\Psi}(\w):=\A(\w)\B_{\Phi}^*(\w),$$
\begin{equation}\label{gw}
(\G_{\Phi,\Psi}(\w))_{i,j}=\langle \tau\phi_j(\w),\tau\psi_i(\w)\rangle\quad\forall\,1\leq j\leq d,i\in I.
\end{equation}
Again we should remark  that $\G_{\Phi,\Psi}(\w)$ is the cross-correlation operator associated to $\tau\Phi(\w)$ and $\tau\Psi(\w)$, that is $\G_{\Phi,\Psi}(\w)=G_{\tau\Phi(\w),\tau\Psi(\w)}$.

\begin{theorem}\label{teinysis}
Let $\Psi=\{\psi_i\}_{i\in I}$ be such that $E(\Psi)$ is a Bessel sequence in $L^2(\R^n)$, $S$ an FSIS generated by a finite set $\Phi$, and $A$ the sampling operator associated to $E(\Psi)$, then
the following are equivalent:
\begin{enumerate}
\item [i)] $\Psi$ provides a one-to-one sampling operator for $S.$
\item [ii)] $\ker(\G_{\Phi,\Psi}(\w))=\ker(\B_{\Phi}^*(\w))$ for a.e. $\w \in \mathbb{T}^n.$
\item [iii)] $\dim(\textnormal{range}(\G_{\Phi,\Psi}(\w)))=\dim_{S}(\w)$ for a.e. $\w \in \mathbb{T}^n.$
\end{enumerate}
\end{theorem}

For the proof of Theorem \ref{teinysis} we need the following.
\begin{lemma}\label{leinysis}
Let $S$ be an FSIS$, \Psi=\{\psi_i\}_{i\in I}$  such that $E(\Psi)$ is a Bessel sequence in $L^2(\R^n)$,  and $A$ the sampling operator associated to $E(\Psi)$.
Then $A$ is one-to-one on $S$ if and only if  $\A(\w)$ is one-to-one on $J_S(\w)$ for a.e. $\w\in\mathbb{T}^n.$
\end{lemma}
\begin{proof}
Since S is a SIS, by Proposition \ref{reb}, $A$ is one-to-one on $S$ if and only if
\begin{equation}\label{aa}
S=\sn\,E(P_S\Psi).
\end{equation}

By Proposition \ref{fun}, equation (\ref{aa}) is equivalent to
\begin{equation}\label{L2}
J_S(\w)=\sn\{\tau (P_S\psi_i)(\w): i\in I\}\quad\mbox{ for a.e. }\w\in\mathbb{T}^n.
\end{equation}
So, we have proved that  $A$ is one-to-one on $S$ if and only if
(\ref{L2}) holds.

On the other side, given $\w\in\mathbb{T}^n$, and using Proposition \ref{reiny} for the sampling operator $\A(\w)$
 and the space $\mathcal{H}=\ell^2(\Z^n)$, we have that $\A(\w)$ is one-to-one on
$J_S(\w)$ if and only if
$$J_S(\w)=\sn\{ P_{J_S(\w)}(\tau\psi_i(\w)): i\in I\}.$$
Then, using Proposition \ref{proy}, we conclude that (\ref{L2}) holds if and only if $\A(\w)$ is one-to-one on
$J_S(\w)$, for a.e. $\w\in\mathbb{T}^n$, which completes the proof of the lemma.

\end{proof}

\begin{proof}[{\it Proof of Theorem \ref{teinysis}}]
Since $\Phi$ is a set of generators for  $S$,  we have that
for a.e. $\w\in \mathbb{T}^n$, $\tau\Phi(\w)$ is a set of generators for  $J_S(\w)$.

Now, for a.e. $\w\in \mathbb{T}^n$ we can apply Theorem \ref{teinys} for the sampling operator $\A(\omega)$
and the finite-dimentional subspace $J_S(\w)$ to obtain the equivalence of the following propositions:

\begin{enumerate}
\item [a)] $\A(\w)$ is one-to-one on
$J_S(\w).$
\item [b)] $\ker(\G_{\Phi,\Psi}(\w))=\ker(\B_{\Phi}^*(\w)).$
\item [c)] $\dim(\text{range}(\G_{\Phi,\Psi}(\w)))=\dim(J_{S}(\w))=\dim_{S}(\w).$
\end{enumerate}
From here the proof follows using Lemma \ref{leinysis}.

\end{proof}

Note that with the previous theorem we have conditions for $A$ to be one-to-one on $\overline{S}_{\g,\ti}$,
and since
$$S_{\g,\ti}=S_{\g}+S_{\ti}\subset \overline{S}_{\g,\ti},$$
we obtain the following corollary.

\begin{corollary}\label{cor-iny-sis}
Let $E(\Psi)$ be a Bessel sequence in $L^2(\R^n)$ for some set of functions $\Psi$. For every $\g,\ti\in\Gamma$, let $\Phi_{\g,\ti}$ be a finite set of generators for $\overline{S}_{\g,\ti}$.
If for each $\g,\ti\in\Gamma,$
$$
\dim(\textnormal{range}(\G_{\Phi_{\g,\ti},\Psi}(\w)))=\dim_{\overline{S}_{\g,\ti}}(\w)\quad\text{for a.e. }\w \in \mathbb{T}^n,
$$ then $A$ is one-to-one on $\chi$.
\end{corollary}

\begin{remark}
It is important to note that the injectivity of $A$ on
$S_{\g,\ti}$ does not imply the injectivity on $\overline{S}_{\g,\ti}$, thus, we have only obtained sufficient
conditions for $A$ to be one-to-one.
This is not a problem in general, because  as we will see in the next section,
stability implies injectivity  in the case of  the sampling operator and stability
is a common and needed assumption in most  sampling applications.
\end{remark}

\subsection{The stability condition}

As a consequence of Proposition \ref{lemstab}, we will obtain
necessary and sufficient conditions for the stability of $A$.

As in the previous subsection, using that the orthogonal projection onto a SIS commutes with integer translates, we have the following
version of Proposition \ref{reiny}.

\begin{proposition}\label{2}
Given S a SIS of $L^2(\R^n)$, $\Psi=\{\psi_i\}_{i\in I}$ such that $E(\Psi)$ is a Bessel sequence in $L^2(\R^n)$ and $A$ the  sampling operator associated to $E(\Psi)$. Then the following are equivalent:
\begin{enumerate}
\item [i)] The sampling operator $A$ is stable for  $S$ with constants $\alpha$ and $\beta$.
\item [ii)] $E(P_S\Psi)$ is a frame for $S$ with constants $\alpha$ and $\beta$.
\end{enumerate}
\end{proposition}

Now we are able to state the stability theorem. We will use the operator related to the fibers, defined by
(\ref{aw}), (\ref{b*w}) and (\ref{gw}).

\begin{theorem}\label{stasub}
Let $\Psi=\{\psi_i\}_{i\in I}$ be such that $E(\Psi)$ is a Bessel sequence for $L^2(\R^n)$ and $A$ the  sampling operator associated to $E(\Psi)$.
Let $S$ be an FSIS, and  $\Phi$ a finite set of functions such that $E(\Phi)$ forms a Parseval frame for $S.$

Then
$E(\Psi)$ provides a stable sampling operator for $S$
if and only if
\begin{enumerate}
\item [i)] $\dim(\textnormal{range}(\G_{\Phi,\Psi}(\w)))=\dim_{S}(\w)$ for a.e. $\w\in \mathbb{T}^n$ and
\item [ii)] There exist constants $0<\alpha\leq\beta<\infty$ such that $$\sigma^2(
\G_{\Phi,\Psi}(\w))\subseteq\{0\}\cup[\alpha,\beta]\quad\mbox{for a.e. }
\w \in \mathbb{T}^n.$$
\end{enumerate}
\end{theorem}

\begin{proof}
$\Phi$ is a Parseval frame for $S$, so, by Theorem \ref{grami3}, we have that
 for a.e. $\w\in \mathbb{T}^n$, $\tau\Phi(\w)$ is a Parseval frame for $J_S(\w)$. Since
 $J_S(\w)$ is a finite-dimensional space of $\ell^2(\Z^n)$, Theorem
\ref{teosta} holds for $\A(\w)$.

So, we only have to prove that A is stable for $S$  with constants $\alpha$ and $\beta$ if and only if $\A(\w)$ is stable for
$J_{S}(\w)$
with constants $\alpha$ and $\beta$.

By Proposition \ref{2}, the stability of $A$ in $S$ is equivalent to
$E(P_{S}\Psi)$ being a frame for $S$ with constants $\alpha$ and $\beta$.
By Theorem \ref{grami3}, this is equivalent to $$\{\tau(P_{S} \psi _i)(\w)\}_{i\in I}$$ being a frame for $J_{S}(\w)$ with constants $\alpha$ and $\beta$ for a.e. $\w\in \mathbb{T}^n.$

On the other hand, given $\w\in \mathbb{T}^n$, the operator $\A(\w)$ is stable for
$J_{S}(\w)$, if and only if
$$\{P_{J_{S}(\w)}(\tau\psi_i(\w))\}_{i\in I}$$ is a frame for $J_{S}(\w)$ with constants $\alpha$ and $\beta$.

The proof can be finished now using  first Proposition \ref{proy}, i.e.
$$\tau (P_{S}\psi_i)(\w)=P_{J_{S}(\w)}(\tau\psi_i(\w))\quad\mbox{ for a.e. } \w\in\mathbb{T}^n,$$
and then Theorem \ref{teosta}.
\end{proof}

Now we apply  Theorem \ref{stasub} and Proposition \ref{lemstab} to obtain the following.

\begin{theorem}\label{stasis}
Let $\Psi=\{\psi_i\}_{i\in I}$ such that $E(\Psi)$ is a Bessel sequence for $L^2(\R^n)$, and for every $\g,\ti\in\Gamma$ let $\Phi_{\g,\ti}$ be a Parseval
frame for $\overline{S}_{\g,\ti}$. Then
$E(\Psi)$ provides a stable sampling operator for $\chi$
if and only if
\begin{enumerate}
\item [i)] $\dim(\textnormal{range}(\G_{\Phi_{\g,\ti},\Psi}(\w)))=\dim_{\overline{S}_{\g,\ti}}(\w)$ for a.e. $\w\in \mathbb{T}^n,\,\forall\,
\g,\ti\in\Gamma$ and
 \item[ii)]There exist constants $0<\alpha\leq\beta<\infty$ such that $$\sigma^2(
\G_{\Phi_{\g,\ti},\Psi}(\w))\subseteq\{0\}\cup[\alpha,\beta]\quad\mbox{for a.e. }
\w \in \mathbb{T}^n,\,\forall\,
\g,\ti\in\Gamma.$$
\end{enumerate}
\end{theorem}

Finally, as in \cite{DL08}, we  obtain a lower bound for the amount of samples. In contrast to the previous section, we only find bounds for stable
operators. We can not say anything about one-to-one operators since we only obtained sufficient conditions for the injectivity.

\begin{proposition}
If the operator A is stable for $\chi$ and $I$ is finite, then
$$\# I\geq \sup_{\g,\ti\in \Gamma}(\textnormal{len}(\overline{S}_{\g,\ti})).$$
\end{proposition}

\begin{proof}
Since $I$ is finite, it holds that $\textnormal{range}(\G_{\Phi_{\g,\ti},\Psi}(\w)))\subset \C^{\#I}$ for a.e.
$\w\in \mathbb{T}^n$.
Hence, by Theorem \ref{stasis}, we have that
$$\dim_{\overline{S}_{\g,\ti}}(\w)=\dim(\text{range}(\G_{\Phi_{\g,\ti},\Psi}(\w)))\leq \#I\quad\mbox{ for a.e. }\w\in \mathbb{T}^n,\,\forall\,
\g,\ti\in\Gamma.$$
This shows that, given $\g,\ti\in\Gamma,$
$$\mbox{ess-sup }\{\dim_{\overline{S}_{\g,\ti}}(\w)\,:\,\w\in\mathbb{T}^n\}\leq\# I.$$
The proof of the proposition follows using \cite[Theorem 3.5]{BDVR94}.

\end{proof}

We would like to note that based in our results, it is possible to state conditions for the injectivity and stability for the sampling operator
in a union of SISs which are not necessarily finitely-generated. For this, condition iii) of Theorem \ref{teinysis}
should be replaced by condition ii).

\section{Necessary and sufficient conditions for the closedness of the sum of two SIS.}\label{section6}

In this section we review the conditions for the sum of two closed subspaces of a Hilbert space to
be closed. Then we apply the results to the class of FSISs in $L^2(\R^n)$.

\subsection{Closedness of the sum of two subspaces of a Hilbert space.}

Throughout this section $\mathcal{H}$ will be a separable Hilbert space.\\
If $S$ is a closed subspace of $\mathcal{H}$, we write $P_S$ for the orthogonal projection onto $S$.
Given $U$ and $V$ closed subspaces
of $\mathcal{H}$, we will use the symbol $P_U\big\vert_V$ to denote the restriction of $P_U$ to the subspace $V$.\\
The orthogonal complement of $U\cap V$ in $U$ will be denoted by
\begin{equation}\label{utilde}
U\ominus V:=U\cap(U\cap V)^{\perp}.
\end{equation}
If $A:\mathcal{H}\rightarrow \mathcal{H}$ is a bounded linear operator, we will use the norm
$$\|A\|=\sup_{x\neq 0}\frac{\|Ax\|}{\|x\|}.$$

The conditions on the closedness of the sum of two closed subspaces of $\mathcal{H}$ will be given in terms of the angle between the subspaces. We refer the reader to \cite{Deu95} for details and proofs.

\begin{definition}
Let $U$ and $V$ be closed subspaces of $\mathcal{H}$.
\begin{enumerate}
\item [a)] The  \textit{minimal angle}  between $U$ and $V$ (or \textit{Dixmier angle}) is the angle in $[0,\frac{\pi}{2}]$ whose cosine is $$\textbf{c}_0[U,V]:=\sup \{|\langle u,v\rangle|\,:\,u\in U, v\in V, \|u\|\leq 1, \|v\|\leq1\}.$$
\item [b)] The \textit{angle} between $U$ and $V$ (or \textit{Friedrichs angle}) is the angle in $[0,\frac{\pi}{2}]$ whose cosine is $$\textbf{c}[U,V]:=\sup\{|\langle u,v \rangle|\,:\,u\in
U\ominus V, v\in V\ominus U\text{ and }\|u\|\leq 1, \|v\|\leq 1\}.$$
\end{enumerate}
\end{definition}

Now we  state some known results concerning both notions of angles between subspaces.

\begin{proposition}\label{cerrados1}
Let $U$ and $V$ be closed subspaces of $\mathcal{H}$.
\begin{enumerate}
\item [i)]$\textbf{c}_0[U,V]=\|P_U\big\vert_V\|.$
\item[ii)]  $\textbf{c}[U,V]=\textbf{c}_0[U\ominus V,V\ominus U].$
\item [iii)] $U+V$ is closed if and only if $\textbf{c}[U,V]<1.$
\end{enumerate}
\end{proposition}

Let $\mathcal{H}$ and $\mathcal{K}$ be separable Hilbert spaces, and $T:\mathcal{H}\rightarrow \mathcal{K}$ a bounded linear operator with closed range. We denote by  $T^{\dagger}$,  the pseudo-inverse of $T$ (see \cite{Chr03} for definition and properties).

The following theorem provides a formula which gives an easy way of calculating the angle between two subspaces using operators associated to frames. Its proof follows from  \cite[Theorem 2.1]{KLL06} and part ii) of Proposition \ref{cerrados1}.

\begin{theorem}\label{teo-c}
Let $U$ and $V$ be closed subspaces of $\mathcal{H}$. Suppose that $X$ and $X'$ are countable subsets of $\mathcal{H}$ which are frames for $U\ominus V$ and $V\ominus U$ respectively. Then,
$$\textbf{c}[U,V]=\|(G_{X'}^{\dagger})^{\frac{1}{2}}G_{X,X'}(G_{X}^{\dagger})^{\frac{1}{2}}\|,$$
where $G_X$ and $G_{X'}$ are the Gramian operators and $G_{X,X'}$ is the cross-correlation operator.
\end{theorem}

\subsection{Closedness of the sum of two shift-invariant subspaces}

As it was stated in Proposition \ref{cerrados1}, the closedness of the sum of two subspaces depends on the Friedrichs angle between them.
In this section, we  provide an expression  for the Friedrichs angle between  two SISs in terms of the gramians of the generators.
In \cite{KLL06} Kim et al has found a similar expression for the Dixmier angle between two SISs.

The main theorem of this part gives necessary and sufficient conditions for the sum of two SISs to be closed. We first state the theorem and then we apply this result to obtain a more general version of Corollary \ref{cor-iny-sis}. The proof of the theorem will be given at the end of the section.

\begin{theorem}\label{teofinal}
Let $U$ and $V$ be SISs of $L^2(\R^n)$. Suppose that $\Phi, \Phi'$ are sets of functions in $L^2(\R^n)$ such that for a.e. $\w\in\mathbb{T}^n$, $\tau\Phi(\w)$ and $\tau\Phi' (\w)$ are frames for $J_{U\ominus V}(\w)$ and
 $J_{V\ominus U}(\w)$ respectively.
Then, $U+V$ is closed if and only if
\begin{equation}\label{ec10}
\textbf{c}[U,V]=\textnormal{ess-sup } \{\|(\G_{\Phi'}(\w)^{\dagger})^{\frac{1}{2}}\G_{\Phi,\Phi' }(\w)(\G_{\Phi}(\w)^{\dagger})^{\frac{1}{2}}\|\,:\, \w\in\mathbb{T}^n\}<1.
\end{equation}
\end{theorem}

 Note that, if $S=S(\Phi)$ is an FSIS, we have that $\tau\Phi(\w)$ is a frame for $J_S(\w)$  for a.e. $\w\in\mathbb{T}^n$, even though $E(\Phi)$ is not a frame for $S$. Thus, if $U$ and $V$ are FSISs,  condition (\ref{ec10}) can
 be checked on any set of generators of the subspaces $U\ominus V$ and $V\ominus U$. At the end of the section we give an example in which we compute the Friedrichs angle between two FSISs.

In the next theorem we show that imposing certain restrictions on the angle between the subspaces, we obtain
necessary and sufficient conditions for the injectivity of the sampling operator. This gives a more complete version of Corollary \ref{cor-iny-sis} .

\begin{theorem}
Let $\Psi=\{\psi_i\}_{i\in I}$ be such that $E(\Psi)$ is a Bessel sequence in $L^2(\R^n)$ and suppose condition (\ref{ec10}) is satisfied for every $\g,\ti\in\Gamma$. If $\Phi_{\g,\ti}$ is a finite set of generators  for $S_{\g,\ti}$,
the following are equivalent:
\begin{enumerate}
\item [i)]$\Psi$ provides a one-to-one sampling operator for $\chi.$
\item [ii)]$\dim(\textnormal{range}(\G_{\Phi_{\g,\ti},\Psi}(\w)))=\dim_{S_{\g,\ti}}(\w)$ for a.e. $\w \in \mathbb{T}^n, \forall\,\g,\ti\in\Gamma.$
\end{enumerate}
\end{theorem}

\begin{proof}
Since condition (\ref{ec10}) is satisfied for every $\g,\ti\in\Gamma$, it holds that the subspaces $S_{\g,\ti}$
are FSISs. The proof of the theorem follows applying Theorem \ref{teinysis} to these subspaces.

\end{proof}

In what follows we will give some lemmas which will be needed for the proof of Theorem \ref{teofinal}.
The results in these lemmas  are interesting by themselves.

The first lemma uses the notion of range function introduced in Definition \ref{def-range}.

\begin{lemma}\label{lema-rango-int}
Given $U$ and $V$ SISs  of $L^2(\R^n)$. Then the range function
$$
R:\mathbb{T}^n\rightarrow\{\text{closed subspaces of } \ell^2(\Z^n)\},\quad R(\w)=J_U(\w)\cap J_V(\w),
$$ is measurable.
\end{lemma}

\begin{proof}\label{proy-int-med}
Recall that the measurability of $R$ is equivalent to $\w\mapsto P_{J_U(\w)\cap J_V(\w)}$ being measurable.

It is known (see \cite{Neu50}) that given $M$ and $N$ closed subspaces of a separable Hilbert space $\mathcal{H}$, for each $x\in\mathcal{H}$,
$$
P_{M\cap N}(x)=\lim_{n\to +\infty}(P_MP_N)^n(x).
$$
Note that if we have two measurable functions
 $$Q_1,Q_2: \mathbb{T}^n \rightarrow \{\text{orthogonal projections in } \ell^2(\Z^n)\},$$
 then the map $ \w \mapsto Q_1(\w)Q_2(\w)$ is measurable.
 For, let $F$ be an arbitrary measurable function from $\mathbb{T}^n$ into $\ell^2(\Z^n)$.
 Then $$ Q_1(\w)Q_2(\w)(F(\w)) = Q_1(\w)(Q_2(\w)(F(\w))).$$
 By Definition \ref{def-range}, the measurability of $Q_2(\w)$ implies the vector measurability of $Q_2(\w)(F(\w)).$
Since $Q_1(\w)$ is measurable, $ Q_1(\w)Q_2(\w)(F(\w))$ is measurable. What shows that
$ \w \mapsto Q_1(\w)Q_2(\w)$ is measurable.

As a consequence, it holds that for any $n \in \N$ the map $\w\mapsto (P_{J_U(\w)}P_{J_V(\w)})^n$ is measurable, that is, for each $a\in\ell^2(\Z^n)$, $\w\mapsto (P_{J_U(\w)}P_{J_V(\w)})^n(a)$ is measurable.
From here the proof follows using that,
$$
P_{J_U(\w)\cap J_V(\w)}(a)=\lim_{n\to +\infty}(P_{J_U(\w)}P_{J_V(\w)})^n(a).
$$
\end{proof}

With the previous lemma we obtain the following property of the fiber spaces.

\begin{lemma}\label{Js}
Let $U$ and $V$ be SISs of $L^2(\R^n)$. Then, $$J_{U\ominus V}(\omega)=J_U(\omega)\ominus J_V(\omega)\quad\text{ for a.e. }\w\in \mathbb{T}^n.$$
\end{lemma}

\begin{proof}
We will first prove that
\begin{equation}\label{J-int}
J_{U\cap V}(\omega)=J_U(\omega)\cap J_V(\omega)\quad\text{ for a.e. }\w\in \mathbb{T}^n.
\end{equation}
Let $R$ be the measurable range function defined in Lemma \ref{lema-rango-int}.
Since
$$
U\cap V=\{f\in L^2(\R^n)\,:\,\tau f(\w)\in R(\w)\text{ for a.e. }\w \in\mathbb{T}^n\},
$$
it follows that $R$ is the range function associated to the shift-invariant space $U\cap V$, thus
(\ref{J-int}) holds.

Using (\ref{J-int}), the proof of the proposition is straightforward as
$$(J_S(\omega))^{\perp}=J_{S^{\perp}}(\omega)\quad\text{for a.e. }\w\in \mathbb{T}^n,$$ for any shift-invariant space $S$ of $L^2(\R^n)$.
\end{proof}

The next lemma follows the ideas from \cite{BG04}. It states that the angle between two shift-invariant spaces is the essential supremum of the angles between the fiber spaces.

\begin{lemma}\label{lem-ess-sup}
Let $U$ and $V$ be SISs of $L^2(\R^n)$. Then,
$$\textbf{c}[U,V]=\textnormal{ess-sup }\{\textbf{c}[J_U(\w),J_V(\w)]\,:\,\w\in\mathbb{T}^n\}.$$
\end{lemma}

\begin{proof}
Given $f\in V$, by Proposition \ref{proy}, we have for a.e. $\w\in\mathbb{T}^n,$
$$\tau (P_U\big\vert_Vf)(\w)=\tau (P_UP_Vf)(\w)=P_{J_U(\w)}P_{J_V(\w)}(\tau f(\w))=P_{J_U(\w)}\big\vert_{J_V(\w)}(\tau f(\w)).$$
This shows that $P_{J_U(\w)}\big\vert_{J_V(\w)}$ is the range operator corresponding
to the shift-preserving operator $P_U\big\vert_{V}$ in the shift-invariant space $V$. What implies  that
\begin{equation}\label{norma-proy-fibras}
\|P_U\big\vert_V\|=\text{ess-sup }\{\|P_{J_{U}(\w)}\big\vert_{J_{V}(\w)}\|\,:\,\w\in\mathbb{T}^n\}
\end{equation}
(see \cite{Bow00} for the definition and properties of shift-preserving operators).

Using (\ref{norma-proy-fibras}), Proposition \ref{cerrados1} and Lemma \ref{Js}, we obtain
\begin{align*}
\textbf{c}[U,V]&=\textbf{c}_0[U\ominus V,V\ominus U]=\|P_{U\ominus V}\big\vert_{V\ominus U}\|\\
&=\text{ess-sup }\{\|P_{J_{U\ominus V}(\w)}\big\vert_{J_{V\ominus U}(\w)}\|\,:\,\w\in\mathbb{T}^n\}\\
&=\text{ess-sup }\{\|P_{J_{U}(\w)\ominus J_{V}(\w)}\big\vert_{J_{V}(\w)\ominus J_{U}(\w)}\|\,:\,\w\in\mathbb{T}^n\}\\
&=\text{ess-sup }\{\textbf{c}_0[J_{U}(\w)\ominus J_{V}(\w),J_{V}(\w)\ominus J_{U}(\w)]\,:\,\w\in\mathbb{T}^n\}\\
&=\text{ess-sup }\{\textbf{c}[J_{U}(\w),J_{V}(\w)]\,:\,\w\in\mathbb{T}^n\}.
\end{align*}
\end{proof}

With the above results, we are able to prove the main theorem of this section.

\begin{proof}[\textit{Proof of Theorem \ref{teofinal}}]
By Lemma \ref{Js}, it follows that  $\tau\Phi(\w)$ and
$\tau\Phi' (\w)$ are frames for $J_{U}(\w)\ominus J_{V}(\w)$ and $J_{V}(\w)\ominus J_{U}(\w)$ respectively,
for a.e. $\w\in\mathbb{T}^n$.

Thus, using Theorem \ref{teo-c}, we obtain
$$\mathbf{c}[J_{U}(\w),J_{V}(\w)]=\|(\G_{\Phi' }(\w)^{\dagger})^{\frac{1}{2}}\G_{\Phi,\Phi' }(\w)(\G_{\Phi}(\w)^{\dagger})^{\frac{1}{2}}\|\quad\text{for a.e. }\w\in\mathbb{T}^n.$$

Hence, by Lemma \ref{lem-ess-sup},
\begin{equation}\label{last-ec}
\textbf{c}[U,V]=\text{ess-sup }\{\|(\G_{\Phi' }(\w)^{\dagger})^{\frac{1}{2}}\G_{\Phi,\Phi' }(\w)(\G_{\Phi}(\w)^{\dagger})^{\frac{1}{2}}\|\,:\, \w\in\mathbb{T}^n\}.
\end{equation}

The proof of the theorem follows from (\ref{last-ec}) and Proposition \ref{cerrados1}.
\end{proof}

Next we provide an example of two shift-invariant spaces whose sum is not closed. In order to prove that,  we compute the Friedrichs angle  between the subspaces.

\begin{example}\label{ejclau}

Let $\varphi_1\in L^2(\R)$ be given by
$$\hat{\varphi_1}(\w)=
\begin{cases}
\cos(2\pi\w)   &\text{ if } 0\leq \w <1\\
\sin(2\pi\w) &\text{ if } 1\leq \w <2\\
0 &\text{otherwise},
\end{cases}$$
and $\varphi_2, \varphi_3\in L^2(\R)$ satisfying that $\hat{\varphi_2}(\w)=\chi_{[2,3)}(\w)$ and $\hat{\varphi_3}(\w)=\chi_{[3,4)}(\w)$.
Define $U=S(\varphi_1, \varphi_2, \varphi_3)$.\\
Consider now $\varphi_0,\varphi_4\in L^2(\R)$, such that $\hat{\varphi_0}(\w)=\chi_{[0,1)}(\w)$ and $\hat{\varphi_4}(\w)=\chi_{[\frac{5}{2}, \frac{7}{2})}(\w)$, set $V=S(\varphi_0, \varphi_4).$

We will prove that $U+V$ is not closed using Theorem \ref{teofinal}.

Let $\{e_k\}_{k\in\Z}$ be the standard basis for $\ell^2(\Z)$. Then,
 $ \tau\varphi_1(\w)=\cos(2\pi\w)e_0+\sin(2\pi\w)e_1$,  $ \tau\varphi_2(\w)=e_2$,  $ \tau\varphi_3(\w)=e_3$,
  $ \tau\varphi_0(\w)=e_0$,  $ \tau\varphi_4(\w)=e_3\chi_{[0,\frac{1}{2})}(\w)+e_2\chi_{[\frac{1}{2},1)}(\w)$. So, we have that for a.e. $\w\in[0,1)$,
$$J_U(\w)\ominus J_V(\w)=\text{span}\{ \tau\varphi_1(\w), \tau\varphi_5(\w)\}\quad\text{and}\quad J_V(\w)\ominus J_U(\w)=\text{span}\{ \tau\varphi_0(\w)\},$$ where $\hat{\varphi_5}(\w)=\chi_{[2,\frac{5}{2})}(\w)+\chi_{[\frac{7}{2},4)}(\w)$.
Thus, by Lemma \ref{Js}, it follows that $U\ominus V=S(\varphi_1, \varphi_5)$ and $V\ominus U=S(\varphi_0)$.

Let $\Phi:=\{\varphi_1,\varphi_5\}$ and $\Phi':=\{\varphi_0\}$, then
$$
\G_{\Phi'}(\w)=1\quad\G_{\Phi}(\w)=
\begin{pmatrix}
1& 0\\
0 & 1
\end{pmatrix}\quad\text{and}\quad \G_{\Phi,\Phi' }(\w)=(\cos(2\pi\w),  0).
$$

Therefore
\begin{align*}
\textbf{c}[U,V]&=\textnormal{ess-sup } \{\|(\G_{\Phi'}(\w)^{\dagger})^{\frac{1}{2}}\G_{\Phi,\Phi' }(\w)(\G_{\Phi}(\w)^{\dagger})^{\frac{1}{2}}\|\,:\, \w\in[0,1)\}\\
&=\textnormal{ess-sup } \{|\cos(2\pi\w)|\,:\,\w\in[0,1)\}=1.
\end{align*}
Hence, by Theorem \ref{teofinal}, $U+V$ is not closed.

\end{example}

\vspace{13pt}
\centerline{ACKNOWLEDGEMENT}
\vspace{13pt}

\noindent We  thank J.L. Romero  for pointing out reference \cite{Deu95} that originated the content of Section \ref{section6}. We also thank the referee for his comments and suggestions that helped to improve the presentation.

\thispagestyle{empty}


\begin{thebibliography}{99}
\addcontentsline{toc}{chapter}{Bibliograf\'\i a}

\bibitem{ACM04}
A.~Aldroubi, C.~Cabrelli and U.~Molter,
Wavelets on irregular grids with arbitrary dilation matrices
and frame atoms for {$L^2(\R^d)$},
\newblock {\it Appl. Comput. Harmon. Anal.}, Special Issue on Frames II, \textbf{17}, 119-140, 2004.



\bibitem{ACM08} A. Aldroubi, C. Cabrelli and U. Molter, Optimal non-linear models for sparsity and sampling,
{\it J. Fourier Anal. Appl.},  \textbf{5-6}, 793-812,  2008.

\bibitem{BD07} T. Blumensath and M. Davies,
Sampling theorems for signals from the union of linear subspaces, Submitted to IEEE Trans. Inform. Theory on 2007.



\bibitem{BDVR94} C. de Boor, R. A. DeVore and A. Ron, The structure of finitely generated shift-invariant spaces in $L^2(\R^d)$, {\it J. Funct. Anal.},  $\textbf{119}$, 37-78, 1994.

\bibitem{BG04} M. Bownik and G. Garrig\'os, Biorthogonal wavelets, MRA's and shift-invariant spaces, {\it Studia Math.},  \textbf{160}, 231--248, 2004.


\bibitem{Bow00} M. Bownik, The structure of shift-invariant subspaces of $L^2(\R^n)$, {\it J. Funct. Anal.},  \textbf{177}, 282-309, 2000.



\bibitem{Chr03} O. Christensen, {\it An introduction to frames and Riesz bases},  Applied and Numerical Harmonic Analysis, Birkh\"auser Boston, Inc., Boston, MA, 2003.


\bibitem{CRT06} E. J. Cand{\`e}s, J. Romberg and T. Tao, Robust uncertainty principles: Exact signal reconstruction from highly incomplete frequency information, {\it IEEE Trans. Inform. Theory}, \textbf{52}, 489-509, 2006.


\bibitem{CT06} E. J. Cand{\`e}s and T. Tao, Near optimal signal recovery from random projections: Universal encoding strategies?, {\it IEEE Trans. Inform. Theory}, \textbf{52}, 5406-5425, 2006.




\bibitem{Deu95} F. Deutsch, The angle between subspaces in Hilbert space, {\it Approximation theory, wavelets and applications},
(S.P. Singh, editor), Kluwer, Netherlands, 107-130, 1995.



\bibitem{Don06} D. Donoho, Compressed sensing, {\it IEEE Trans. Inform. Theory}, \textbf{52}, 1289-1306, 2006.

\bibitem{Eld08} Y. C. Eldar, Compressed sensing of analog signals, submitted to IEEE Trans. Signal Process. on 2008, 	 arXiv:0806.3332v1.

\bibitem{EM08} Y. C. Eldar and M. Mishali, Robust recovery of signals from a union of subspaces, preprint 2008,	 arXiv:0807.4581v1.

\bibitem{FW01}
H.~G. Feichtinger and T.~Werther, \emph{Atomic systems for subspaces},
      Proceedings SampTA 2001 (L.Zayed, ed.), Orlando, FL, 2001, pp.~163--165.

\bibitem{Hel64} H. Helson, \textit{Lectures on Invariant Subspaces}, Academic Press, New York, 1964.



\bibitem{KLL06} H. O. Kim, R. Y. Lim and J. K. Lim, Characterization of the closedness of the sum of two shift-invariant spaces,  {\it J. Math. Anal. Appl.},   $\textbf{320}$,  381-395, 2006.

\bibitem{Lan67} H. J. Landau, Sampling, data transmission, and the Nyquist rate,  {\it Proceedings of the IEEE},   $\textbf{55}$,  1701-1706, 1967.

\bibitem{LO04} S. Li and  H. Ogawa, Pseudoframes for subspaces with applications, {\it J. Fourier Anal. Appl.}, \text{10}, 4, 409-431, 2004.

\bibitem{DL08} Y. Lu and M. Do, A theory for sampling signals from a union of subspaces, {\it IEEE Trans. Signal Process.}, $\textbf{56}$, 2334-2345, 2008.


\bibitem{Neu50} John von Neumann,  {\it Functional Operators. II. The Geometry of Orthogonal Spaces},
Annals of Mathematics Studies, no. 22. Princeton University Press, Princeton, N. J.,  1950.



\bibitem{Rud91} W. Rudin, {\it Functional Analysis},  International series in pure and applied mathematics,
Mc Graw-Hill, New York, NY, 1991.

\bibitem{RS95} A. Ron and Z. Shen, Frames and stable bases for shift-invariant subspaces of $L^2(\R^d)$, {\it Canad. J. Math.},  \textbf{47}, 1051-1094, 1995.



\bibitem{VMB02} M. Vetterli, P. Marziliano and T. Blu, Sampling signals with finite rate of innovations,  {\it IEEE Trans. Signal Process.},  \textbf{50}, 1417-1428, 2002.



\end{thebibliography}
\end{document}